\documentclass[preprint,12pt,3p]{elsarticle}

\usepackage[utf8]{inputenc}
\usepackage{amssymb}
\usepackage{amsthm}
\usepackage{amsmath}
\usepackage{mathtools}
\usepackage{mathrsfs}
\usepackage[shortlabels]{enumitem}
\usepackage[hidelinks,colorlinks=false]{hyperref}
\usepackage{xcolor}
\usepackage{caption}

\newtheorem{theorem}{Theorem}
\newtheorem{lemma}[theorem]{Lemma}

\newtheorem{remark}[theorem]{Remark}
\newtheorem{corollary}[theorem]{Corollary}
\newtheorem{conjecture}[theorem]{Conjecture}

\newtheorem{question}[theorem]{Question}

\begin{document}

\begin{frontmatter}

\title{On the Chromatic Number of Stable Kneser Hypergraphs: Verifying the Conjecture for New Families}

\author[1]{Hamid Reza Daneshpajouh}
\cortext[cor1]{Corresponding author}
\ead{Hamid-Reza.Daneshpajouh@nottingham.edu.cn}

\address[1]{School of Mathematical Sciences, University of Nottingham Ningbo China, 199 Taikang East Road, Ningbo, 315100, China}

\begin{abstract}
One of the key unsolved conjectures in hypergraph coloring is about the chromatic number of $s$-stable $r$-uniform Kneser hypergraphs $\mathrm{KG}^r(n,k)_{s\textup{-stab}}$. The problem remains largely open, particularly in the case where $s > r\geq 3$. To the best of our knowledge, no information is available except a limited number of computations conducted for the instances when $r=3, 4$, $s=4, 5$, $k=2,3$ with some $n$ does not exceed 14. In this study, we verify the conjecture for infinity many values of the parameters $n$ and $k$. In particular, we demonstrate: (i) the validity of the conjecture for $r = 4$, $s = 6$ under the condition that $3 \mid n$ or $k=2$, and (ii) for $r = 4$, $k = 2$, $s = 5$ given $3 \nmid n$. As far as we are aware, this provides the first rigorous theoretical proof of the conjecture (for the case $s > r\geq 3$) for infinitely many parameter values, extending beyond finite computational verification. Furthermore, our methods rely on a detailed study of vector-stable Kneser graphs, an approach that not only yields these results but also provides a deeper understanding of their chromatic numbers.
\end{abstract}

\begin{keyword}
Vector-stable Kneser graph, stable Kneser hypergraph, chromatic number, Tucker's lemma, hypergraph coloring.
\end{keyword}

\end{frontmatter}

\section{Introduction}

Throughout this note, the symbol $[n]$ is used for the set $\{1,\ldots, n\}$. A subset $X \subseteq [n]$ is called \textit{$s$-stable} if $s\leq |i-j|\leq n-s$ for all distinct $i, j\in X$. The $r$-uniform Kneser hypergraph $\mathrm{KG}^{r}(n,k)$ is a hypergraph whose vertices are all $k$-subsets of $[n]$ and whose hyperedges are all sets $\{A_1, \ldots , A_r\}$ of $r$ vertices where $A_i\cap A_j=\emptyset$ for $i\neq j$. The $r$-uniform $s$-stable Kneser hypergraph $\mathrm{KG}^{r}(n,k)_{s\textup{-stab}}$ is an induced hypergraph of $\mathrm{KG}^{r}(n,k)$ whose vertices are all $s$-stable $k$-subsets of $[n]$. 

In 1978, Lov\'{a}sz~\cite{Lo78} established that
\[
\chi\left(\mathrm{KG}^2(n,k)\right)=n-2(k-1)\quad \text{for all } n\geq 2k.
\] 
Shortly afterward, Schrijver~\cite{Sch78} realized that the chromatic number remains constant even after removal of all non 2-stable vertices, i.e.,
\[
\chi\left(\mathrm{KG}^2(n,k)_{2\textup{-stab}}\right)=n-2(k-1)\quad \text{for all } n\geq 2k.
\]
Expanding on this foundation, Alon--Frankl--Lov\'{a}sz~\cite{alon1986chromatic} found that the results of Lov\'{a}sz can be generalized to $r$-uniform Kneser hypergraphs:
\[
\chi \left(\mathrm{KG}^{r}(n,k)\right)=\left\lceil\frac{n-r(k-1)}{r-1}\right\rceil\quad\text{for all } n\geq rk \text{ and } r\geq 2.
\]
In light of these results, Ziegler~\cite{ziegler2002generalized} conjectured that the results of Alon--Frankl--Lov\'{a}sz remain valid even with the elimination of all non $r$-stable vertices.

\begin{conjecture}[\cite{ziegler2002generalized}]\label{Conj: zig}
If $r\geq 2$ and $n\geq rk$, then
\[
\chi\left(\mathrm{KG}^r(n,k)_{r\textup{-stab}}\right)=\left\lceil\frac{n-r(k-1)}{r-1}\right\rceil.
\]
\end{conjecture}

Furthermore, Meunier~\cite{Me11} expanded this conjecture to $s$-stable $r$-uniform Kneser hypergraphs.

\begin{conjecture}[\cite{Me11}]\label{Conj: Meu}\footnote{It is worth noting that there was an alternative generalization of Ziegler's conjecture~\cite{frick2020chromatic}, which was ultimately found to be incorrect~\cite{daneshpajouh2023counterexample}.}
If $r, s\geq 2$ and $n\geq \max\{r,s\}k$, then
\[
\chi \left(\mathrm{KG}^r(n,k)_{s\textup{-stab}}\right)=\left\lceil\frac{n-\max\{r,s\}(k-1)}{r-1}\right\rceil.
\]
\end{conjecture}

Alon et al.~\cite[Lemma 3.3]{Alon2009} demonstrated that if Conjecture~\ref{Conj: zig} holds for $r_1$ and $r_2$ (for all values of $n$ and $k$), then it is valid for $r_1r_2$. By combining this lemma with the Schrijver result, they succeeded in validating Conjecture~\ref{Conj: zig} when $r$ is a power of $2$~\cite[Corollary 3.4]{Alon2009}. Regarding Conjecture~\ref{Conj: Meu}, significant progress has been made when $s < r$. Notably, Frick~\cite[Theorem 1.1]{frick2020chromatic} confirmed the validity of Conjecture~\ref{Conj: Meu} for $s = 2$ for all $r \geq 2$, as well as for $ r > 6s-6$ when $r$ is a prime power and $s \geq 2$~\cite[Theorem 1.4]{frick2020chromatic}. Additionally, results for $r \geq 4$ ($r\neq 6, 12$) and $s = 3$ have been established under specific conditions on $n$ and $k$~\cite[Theorem 3.8]{frick2020chromatic}. In contrast, when stability exceeds uniformity $s > r$, the existing knowledge is limited except in the graph case. In fact, when $r=2$, Chen~\cite{chen2015multichromatic} showed that Conjecture~\ref{Conj: Meu} holds for all even $s$, and Jonsson~\cite{jonsson2012chromatic} proved it for all $s \geq 4$ provided $n$ is sufficiently large. The current author and his colleagues~\cite{daneshpajouh2021colorings} also validated this conjecture for all $s \geq 2$ in the case where $k = 2$. However, in the case of hypergraphs, when $r\geq 3$, to the best of our knowledge, not much has been known, except limited computational verifications for small parameter values ($r \in \{3,4\}$, $s \in \{4,5\}$, $k \in \{2,3\}$, and $n \leq 14$)~\cite{Me11}. Prior work established that Conjecture~\ref{Conj: Meu} holds when $r$ is a power of 2 and $s$ is a multiple of $r$ as a consequence of the Chen result~\cite[Corollary 5]{chen2015multichromatic} combined with a combinatorial lemma~\cite[Lemma 4]{Me11}. This lemma asserts that if Conjecture~\ref{Conj: zig} holds for $r_1$ (for all $n,k$) and Conjecture~\ref{Conj: Meu} holds for $(r_2,s_2)$ with $s_2 \geq r_2$ (for all $n,k$), then Conjecture~\ref{Conj: Meu} also holds for $(r_1r_2, r_1s_2)$. However, as noted in~\cite{Meunier-website}, the proof of this lemma contains a gap, leaving its validity uncertain.

In this work, we develop an analogue of Meunier's combinatorial lemma by using the concept of vector-stable Kneser hypergraphs (the definition is provided immediately following Corollary~\ref{coro:2}) and imposing precise parameter constraints.
While the fundamental proof strategy remains largely unaltered, this adjustment enables us to circumvent the previously mentioned gap, albeit for a somewhat weaker variant of that lemma. Combining this lemma with established results and recent advances from~\cite{daneshpajouh2021neighborhood, daneshpajouh2021colorings}, we obtain one of the main results of this paper.

\begin{theorem}\label{Thm:main1}
    Conjecture~\ref{Conj: Meu} holds for parameters \( r = 2^{m+1} \) and \( s = 3 \times 2^{m} \) with \( m \geq 1 \), provided
    \[
        n - 3 \times 2^{m}(k-1) \bmod (2^{m+1} - 1) \in \{2^{m}+1, \dots, 2^{m+1}-1\}.
    \]
    Furthermore, when \( k = 2 \), this condition becomes unnecessary; that is, Conjecture~\ref{Conj: Meu} holds unconditionally for \( r = 2^{m+1} \), \( s = 3 \times 2^{m} \), and \( k = 2 \).
\end{theorem}

Specifically, applying Theorem~\ref{Thm:main1} to the case where $m=1$, we obtain the following result.

\begin{corollary}\label{corollary:1}
For all integers $n \geq 6k$, we have
\[
\chi\left(\mathrm{KG}^{4}(n,k)_{6\textup{-stab}}\right) = \left\lceil\frac{n - 6(k-1)}{3}\right\rceil,
\]
provided either $3$ divides $n$ or $k=2$.
\end{corollary}

Moreover, observe that $\mathrm{KG}^{4}(n,2)_{6\textup{-stab}}$ is a subgraph of $\mathrm{KG}^{4}(n,2)_{5\textup{-stab}}$. Therefore, Corollary~\ref{corollary:1} implies
\[
\left\lceil\frac{n-6}{3}\right\rceil \leq \chi\left(\mathrm{KG}^{4}(n,2)_{6\textup{-stab}}\right) \leq \chi\left(\mathrm{KG}^{4}(n,2)_{5\textup{-stab}}\right).
\]
Moreover, it is known from~\cite{Me11} that $\mathrm{KG}^r(n,k)_{s\textup{-stab}}$ admits a coloring using
\[
\left\lceil\frac{n - \max\{r,s\}(k-1)}{r-1}\right\rceil
\]
colors. Consequently, we obtain
\[
\chi\left(\mathrm{KG}^{4}(n,2)_{5\textup{-stab}}\right) \leq \left\lceil\frac{n-5}{3}\right\rceil.
\]
Since 
\[
\left\lceil\frac{n-6}{3}\right\rceil = \left\lceil\frac{n-5}{3}\right\rceil
\]
whenever $3 \nmid n$, we establish the following corollary:

\begin{corollary}\label{coro:2}
For all integers $n \geq 10$ with $3 \nmid n$, we have
\[
\chi\left(\mathrm{KG}^{4}(n,2)_{5\textup{-stab}}\right) = \left\lceil\frac{n-5}{3}\right\rceil.
\]
\end{corollary}

As previously mentioned, an essential component in these advancements is the investigation of the chromatic number of vector-stable Kneser graphs, a class of graphs introduced in~\cite{daneshpajouh2021neighborhood}. To contextualize this work and present our two next results, we require a broader notion of stability. For a $k$-subset $A \subseteq [n]$, let $A(1), A(2), \ldots, A(k)$ denote its elements in increasing order. Given a positive integer vector $\vec{s} = (s_1, \ldots, s_k)$, a $k$ subset $A$ is called \textit{$\vec{s}$ stable} if it satisfies conditions $s_j \leq A(j+1) - A(j)$ for $1 \leq j \leq k-1$ and $A(k) - A(1) \leq n - s_k$. Note that the case $\vec{s} = (s, \ldots, s)$ recovers the standard definition of $s$-stability. The $r$-uniform $\vec{s}$-stable Kneser hypergraph $\mathrm{KG}^r(n,k)_{\vec{s}\textup{-stab}}$ is then defined as the hypergraph whose vertex set consists of all $\vec{s}$-stable $k$-subsets of $[n]$, and whose hyperedges are $r$-tuples of pairwise disjoint such subsets. The graph case ($r=2$) was the focus of~\cite{daneshpajouh2021neighborhood}, whose main result includes the following theorem.

\begin{theorem}[{\cite[Theorem 4]{daneshpajouh2021neighborhood}}]\label{Theorem: Main of Hamid & Jozsef}
Let $n, k$ be positive integers with $k \geq 2$, and let $\vec{s} = (s_1, \ldots, s_k)$ be an integer vector such that $s_i \geq 2$ for $i = 1, \ldots, k-1$, $s_k \in \{1,2\}$, and $n \geq \sum_{i=1}^{k-1} s_i + 2$. Then,
	\[
	\chi\left(\mathrm{KG}^{2}(n, k)_{\vec{s}\textup{-stab}}\right) = n - \sum_{i=1}^{k-1} s_i.
	\] 
\end{theorem}

Then the authors naturally posed the following question:

\begin{question}[{\cite[Question 16]{daneshpajouh2021neighborhood}}]\label{Question: Main of Hamid & Jozsef}
Let $\vec{s} = (s_1, \ldots, s_k)$ be an arbitrary positive integer vector and $n \geq \sum_{i=1}^{k} s_i$. What is the chromatic number of $\mathrm{KG}^{2}(n, k)_{\vec{s}\textup{-stab}}$ in terms of $n$, $k$, and $\vec{s}$?
\end{question}

Our next result provides a partial answer to this question. In particular, it shows that when $s_k$ exceeds the minimum of the parameters $s_1, \ldots, s_{k-1}$, the chromatic number of $\mathrm{KG}(n,k)_{\vec{s}\text{-stab}}$ deviates from the expression $n - \sum_{i=1}^{k-1} s_i$. The precise behavior is captured in the following theorem.

\begin{theorem}\label{Theorem:Main2}
Let $k \geq 2$, and let $\vec{s} = (s_1, \ldots, s_k)$ be a vector of positive integers such that $s_i \geq 2$ for all $1 \leq i \leq k-1$. Set $m = \min\{s_1, \ldots, s_{k-1}\}$ and assume $s_k \leq 2m$. If $n \geq \sum_{i=1}^{k} s_i$ and either
\begin{enumerate}
    \item[(i)] $m = 2$, or
    \item[(ii)] $s_i$ is even for every $1 \leq i \leq k-1$,
\end{enumerate}
then
\[
\chi\left(\mathrm{KG}^{2}(n, k)_{\vec{s}\textup{-stab}}\right) = n - \sum_{i=1}^{k-1} s_i - \max\left\{0, s_k - m\right\}.
\]
\end{theorem}

Theorem~\ref{Theorem:Main2} might lead one to conjecture that the formula
\[
\chi\left(\mathrm{KG}^{2}(n, k)_{(s_1,\ldots, s_{k-1}, s_k)\textup{-stab}}\right) = n - \sum_{i=1}^{k-1} s_i - \max\left\{0, s_k - m\right\}
\]
holds generally for all $s_k$. The following theorem shatters this intuition, proving the conjecture false. Indeed, it shows that $s_k$ can be increased beyond $2m$ up to a certain threshold without altering the chromatic number.
\begin{theorem}\label{Theorem: Main3}
Let $n$ and $s_1$ and $s_2$ be positive integers such that  $2s_1\leq s_2\leq 3s_1$ and $n \geq 2s_2-2$. Then,
\[
\chi\left(\mathrm{KG}^{2}(n, 2)_{(s_1, s_2)\textup{-stab}}\right) = n - 2s_1.
\]
\end{theorem}

\section*{Acknowledgment} The author would like to thank Fr{\'e}d{\'e}ric Meunier for his valuable comments on an earlier draft of this paper which enhanced the clarity and presentation of the work.

\section{Proof of the main results}\label{proof}

\subsection{Proof of Theorem~\ref{Thm:main1}}
As an initial step, let us present a similar upper bound for the chromatic number of vector-stable Kneser hypergraphs. For convenience, we will define the chromatic number of the null graph (the graph that contains no vertices) to be $-\infty$.

\begin{lemma}\label{lemma: coloring}
For any vector $\vec{s}=(s_1,\ldots, s_k)$ of positive integers, natural numbers $r\geq 2$, and $k\geq 1$ we have
\[
\chi (\mathrm{KG}^{r}(n,k)_{\vec{s}\textup{-stab}})\leq\left\lceil\frac{n-\max\{r(k-1),\sum_{i=1}^{k-1}s_i\}}{r-1}\right\rceil.
\]
\end{lemma}

\begin{proof}
The hypergraph $\mathrm{KG}^{r}(n,k)_{\vec{s}\textup{-stab}}$ is the sub-hypergraph of the usual $r$-uniform hypergraph $\mathrm{KG}^{r}(n,k)$. It is well known that $\mathrm{KG}^{r}(n,k)$ admits a proper coloring with $\left\lceil\frac{n-r(k-1)}{r-1}\right\rceil$ number of colors. So, on the one hand we have
\[
\chi (\mathrm{KG}^{r}(n,k)_{\vec{s}\textup{-stab}})\leq\left\lceil\frac{n-r(k-1)}{r-1}\right\rceil.
\]
On the other hand, let $n-(\sum_{i=1}^{k-1}s_i)=q(r-1)+\beta$ where $q\geq 0$ and $0\leq\beta < r-1$. Set 
$A_i=\{(i-1)(r-1)+1, \ldots, i(r-1)\}$ for $i=1, \ldots, q$, and add one other class $A_{q+1}=\{q(r-1)+1, \ldots, q(r-1)+\beta\}$ if $\beta > 0$. First note that the minimum element of an $\vec{s}$-stable $k$-subset of $[n]$ cannot be larger than $n-\sum_{i=1}^{k-1}s_i$, and hence the following map is well defined
\begin{align*}
  c :  & V(\mathrm{KG}^{r}(n,k)_{\vec{s}\textup{-stab}})\longrightarrow \{1, \ldots, q+1\}\\
  & B\longmapsto \min\{i : A_i\cap B\neq\emptyset\}.
\end{align*}
Moreover, this map defines a proper coloring for $\mathrm{KG}^{r}(n,k)_{\vec{s}\textup{-stab}}$ as the size of each $A_i$ is less than or equal $r-1$. Thus,
\[
\chi (\mathrm{KG}^{r}(n,k)_{\vec{s}\textup{-stab}})\leq\left\lceil\frac{n-(\sum_{i=1}^{k-1}s_i)}{r-1}\right\rceil,
\]
which completes the proof.
\end{proof}

We will now introduce the central lemma of this part.

\begin{lemma}\label{Lemma: Main} 
Let $r_1, r_2 \geq 2$, ${k'} \geq 1$, and $s_1, \ldots, s_{k'}$ be positive integers such that $s_1, \ldots, s_{k'-1} \geq r_2$ and $s_{k'}\leq r_2+1$. Assume the following conditions hold:
\begin{enumerate}[(i)]
    \item For all integers $k$ and $n$ with $n \geq r_1 k$, 
    \[ \chi \left(\mathrm{KG}^{r_1}(n,k)_{r_1\textup{-stab}}\right) \geq \frac{n - r_1 (k-1)}{r_1 - 1}. \]
    \item For all integers $n$ with $n \geq s_1 + \cdots + s_{k'}$, it holds that
    \[ \chi \left(\mathrm{KG}^{r_2}(n,k')_{(s_1, \ldots, s_{k'})\textup{-stab}}\right) \geq \frac{n - \left( \sum_{i=1}^{k'-1} s_i \right)}{r_2 - 1}. \]
\end{enumerate}
Then, for all $n \geq r_1(s_1 + \cdots + s_{k'})$, the following inequality is satisfied: 
\[ \chi (\mathrm{KG}^{r}(n,k')_{\vec{s}\textup{-stab}}) \geq \frac{n - r_1 \left( \sum_{i=1}^{k'-1} s_i \right)}{r - 1}, \] 
where $r = r_1 r_2$ and $\vec{s} = (r_1 s_1, \ldots, r_1 s_{k'})$. 
\end{lemma}

\begin{proof}
Let \(n \geq \max\left\{(r-1)(t-1) + r_1\left(\sum_{i=1}^{k^{\prime}-1}s_i\right) + r,\ r_1(s_1 + \cdots + s_{k'})\right\}\). We note that the inequality \((r-1)(t-1) + r_1\left(\sum_{i=1}^{k^{\prime}-1}s_i\right) + r \geq r_1(s_1 + \cdots + s_{k'})\) holds in almost all cases, as \(s_{k'} \leq r_2 + 1\). The only exception occurs when \(t=1\) (and \(s_{k'}=r_2+1\)). Therefore, to establish the result in both scenarios (\(t=1\) and \(t>1\)), it suffices to prove that \(\chi \left(\mathrm{KG}^{r}(n,k^{\prime})_{\vec{s}\textup{-stab}}\right) > t\). If $t = 1$, then $\mathrm{KG}^{r}(n,k')_{\vec{s}\textup{-stab}}$ is not $1$-colorable because the vertices 

\[
\begin{array}{c}
\{1, s_1+1, s_1+s_2+1, \ldots, \textstyle\sum_{i=1}^{k'-1}s_i+1\}, \\
\{2, s_1+2, s_1+s_2+2, \ldots, \textstyle\sum_{i=1}^{k'-1}s_i+2\}, \\
\vdots \\
\{r, s_1+r, s_1+s_2+r, \ldots, \textstyle\sum_{i=1}^{k'-1}s_i+r\},
\end{array}
\]
form a hyperedge. Now, consider the case $t \geq 2$. Suppose, contrary to our claim, that $\chi (\mathrm{KG}^{r}(n,k^{\prime})_{\vec{s}\textup{-stab}})\leq t$. Assume
\[
c : V\left(\mathrm{KG}^{r}(n,k^{\prime})_{\vec{s}\textup{-stab}}\right)\to\{1,\ldots, t\}
\]
be a proper $t$-coloring of $\mathrm{KG}^{r}(n,k^{\prime})_{\vec{s}\textup{-stab}}$. Set 
\[
m= (r_2-1)(t-1)+ \left(\sum_{i=1}^{k^{\prime}-1}s_i\right) + r_2.
\]
We define a $t$-coloring of $\mathrm{KG}^{r_1}(n,m)_{r_1\textup{-stab}\textup{-stab}}$ as follows. Take an 
\[
A=\{a_1,\ldots, a_m\}\in V(\mathrm{KG}^{r_1}(n,m)_{r_1\textup{-stab}}).
\] 
Consider the map $h: V(\mathrm{KG}^{r_2}(m,k^{\prime})_{(s_1,\ldots, s_{k^{\prime}})\textup{-stab}})\to [t]$ which sends each $B$ to $c(\{a_i : i\in B\})$. This map is not a proper coloring as:
\[
m-\left(\sum_{i=1}^{k^{\prime}-1}s_i\right)= (r_2-1)(t-1) + r_2 \quad \& \quad m\geq s_1+\cdots +s_{k^{\prime}}\,\, (\text{as} \,\,t, r_2\geq 2\,\, \text{and}\,\, s_{k^{\prime}}\leq r_2+1 ).
\]
So, there are $r_2$ pairwise disjoint $k^{\prime}$-subsets $B_A^1, \ldots, B_A^{r_2}\subseteq A$ such that $c(B_A^1)=\cdots =c(B_A^{r_2})= c_{\star}$. Now, define the color of $A$ as their common color, i.e., $c_{\star}$. Do the same procedure for all other vertices of $\mathrm{KG}^{r_1}(n,m)_{r_1\textup{-stab}}$. This coloring is not proper as well, since
\begin{align*}
  n-r_1(m-1) & \geq \left((r-1)(t-1)+r_1\left(\sum_{i=1}^{k^{\prime}-1}s_i\right)+ r\right) \\
  & \quad - \left((r-r_1)(t-1)+r_1\left(\sum_{i=1}^{k^{\prime}-1}s_i\right)+r-r_1\right) \\
  & = (r_1-1)(t-1)+r_1 \quad \& \quad n\geq r_1m,
\end{align*}
which implies $\chi(\mathrm{KG}^{r_1}(n,m)_{r_1\textup{-stab}}) > t$, by our hypothesis. Hence, there exist pairwise disjoint $A_1, \ldots , A_{r_1}\in V(\mathrm{KG}^{r_1}(n,m)_{r_1\textup{-stab}})$ such that all them have the same color. Thus,
\[
E=\{B_{A_i}^{j} : 1\leq i\leq r_1, 1\leq j\leq r_2 \}
\] 
is a monochromatic edge in $\mathrm{KG}^{r}(n,k^{\prime})_{\vec{s}\textup{-stab}}$. This contradicts that $c$ is a proper coloring of $\mathrm{KG}^{r}(n,k^{\prime})_{\vec{s}\textup{-stab}}$. 
\end{proof}

\begin{proof}[\textbf{Proof of Theorem~\ref{Thm:main1}}]
By Theorem~\ref{Theorem: Main of Hamid & Jozsef}, for $k\geq 2$, $n\geq\sum_{i=1}^{k-1}s_i+2$, and $\vec{s}=(s_1, \ldots, s_k)$ with $s_i\geq2$ for $i\neq k$ and $s_k=2$, we have
\[
\chi\left(\mathrm{KG}(n, k)_{\vec{s}\textup{-stab}}\right)= n-\sum_{i=1}^{k-1}s_i.
\]
In contrast, it is known~\cite{Alon2009} that
\[
\chi\left(\mathrm{KG}^{2^{m}}(n,k)_{2^{m}\textup{-stab}}\right)= \left\lceil\frac{n-2^{m}(k-1)}{2^m-1}\right\rceil.
\]
Applying Lemma~\ref{lemma: coloring} and Lemma~\ref{Lemma: Main} yields
\[
\chi \left(\mathrm{KG}^{2^{m+1}}(n,k)_{(s_1\times 2^{m},\ldots,s_{k-1}\times 2^{m}, 2^{m+1})\textup{-stab}}\right)=\left\lceil\frac{n-2^{m}\left(\sum_{i=1}^{k-1}s_i\right)}{2^{m+1}-1}\right\rceil.
\]
Specializing to $s_1=\cdots=s_{k-1}=3$ gives
\[
\chi \left(\mathrm{KG}^{2^{m+1}}(n,k)_{(3\times 2^{m},\ldots,3\times 2^{m}, 2^{m+1})\textup{-stab}}\right)=\left\lceil\frac{n-3\times 2^{m}(k-1)}{2^{m+1}-1}\right\rceil.
\]
Since $\mathrm{KG}^{2^{m+1}}(n-2^m,k)_{(3\times 2^{m},\ldots,3\times 2^{m}, 2^{m+1})\textup{-stab}}$ is a subhypergraph of $\mathrm{KG}^{2^{m+1}}(n,k)_{3\times 2^{m}\textup{-stab}}$, it follows that
\[
\left\lceil\frac{n-3\cdot 2^{m}(k-1)-2^{m}}{2^{m+1}-1}\right\rceil \leq \chi\left(\mathrm{KG}^{2^{m+1}}(n,k)_{3\times 2^{m}\textup{-stab}}\right).
\]
The desired result then follows from Lemma~\ref{lemma: coloring}, noting that $\left\lceil\frac{n-3\cdot 2^{m}(k-1)-2^{m}}{2^{m+1}-1}\right\rceil = \left\lceil\frac{n-3\cdot 2^{m}(k-1)}{2^{m+1}-1}\right\rceil$ when the remainder of $n-3\cdot 2^{m}(k-1)$ modulo $(2^{m+1}-1)$ lies in $\{2^m, \ldots, 2^{m+1}-2\}$. The case $k=2$ follows directly from Lemma~\ref{Lemma: Main}, as Conjecture~\ref{Conj: zig} is confirmed for $r=2^m$~\cite{Alon2009} and Conjecture~\ref{Conj: Meu} holds for $k=2$~\cite{daneshpajouh2021colorings}.
\end{proof}

\subsection{Proof of Theorem~\ref{Theorem:Main2}}

The proof of Theorem~\ref{Theorem:Main2} is based on several steps. We first state the necessary lemmas.

\begin{lemma}\label{Lemma: 1 coloring}
  Let $n, k$ be positive integers with $k \geq 2$ and $n \geq \sum_{i=1}^{k} s_i$. If $\vec{s}=(s_1, \ldots, s_k)$ is a vector of positive integers with $s_{k}\leq 2\times\min\{s_i: 1\leq i\leq k-1\}$, then
  \[
   \chi\left(\mathrm{KG}(n, k)_{\vec{s}\textup{-stab}}\right)\leq n - \sum_{i=1}^{k-1} s_i - \max\{0, s_k-\min\{s_1, \ldots, s_{k-1}\} \}.
  \]
\end{lemma}

\begin{proof}
  Let $m = \min\{s_1, \ldots, s_{k-1}\}$ and define $\alpha = s_k - m$. If $\alpha \leq 0$, then $s_k \leq m$ and the result holds by Lemma~\ref{lemma: coloring}. We now assume $\alpha > 0$; note that we also have $\alpha \leq m$ by our assumption. We need to provide a proper coloring of $\mathrm{KG}(n, k)_{\vec{s}\text{-stab}}$ using $n - \left(\sum_{i=1}^{k-1} s_i\right) - \alpha$ colors. Without loss of generality, let us assume $s_1 = m$. Indeed, should the minimum occur for an alternative index $j$, $s_j = m$, applying the subsequent homomorphism \begin{align*}
    \phi : \mathrm{KG}(n,k)_{\vec{s}\textup{-stab}} &\longrightarrow \mathrm{KG}(n - \sum_{i=1}^{j-1} s_i, k - j + 1)_{\vec{s}^* = (s_j, \ldots, s_k)} \\
    B = \{b_1 < \cdots < b_k\} &\longmapsto B' = \{b_j - \sum_{i=1}^{j-1} s_i, \ldots, b_k - \sum_{i=1}^{j-1} s_i\},
  \end{align*}the desired result is obtained upon verifying our claim for the scenario where the minimum occurs at the initial index. Consequently, establishing the formula for the case $s_1 = m$ is sufficient. Now, define the coloring
  \begin{align*}
    c : V\left(\mathrm{KG}(n, k)_{\vec{s}\textup{-stab}}\right) &\longrightarrow \left\{s_1 + 1, \ldots, n - \left(\sum_{i=2}^{k} s_i - s_1\right)\right\} \\
    B &\longmapsto \min\left\{ i \in B : s_1 + 1 \leq i \leq n - \left(\sum_{i=2}^{k} s_i - s_1\right) \right\}.
  \end{align*}
 In order to demonstrate that this is indeed a proper coloring, we need to show that the range of $c$ is well-defined; that is, every $\vec{s}$-stable set $A$ must include an element from the interval $I = \left[s_1 + 1, n - \left(\sum_{i=2}^{k} s_i - s_1\right)\right]$. Assume, for the sake of contradiction, that a $\vec{s}$-stable set $A = \{a_1 < \cdots < a_k\}$ is entirely contained within the complement of $I$, which corresponds to the set 
 \[
  \Gamma = \{1, \ldots, s_1\} \cup \left\{n - \left(\sum_{i=2}^{k} s_i - s_1\right) + 1, \ldots, n\right\}.
  \]
Given that $A$ is $\vec{s}$-stable, it follows that $a_2 \geq a_1 + s_1 \geq s_1 + 1$ (as $a_1 \geq 1$). Consequently, this leads to $a_2, \ldots, a_k \notin \{1, \ldots, s_1\}$. On the other hand, it is not feasible that the upper section of $\Gamma$, $J= \left\{n - \left(\sum_{i=2}^{k} s_i - s_1\right) + 1, \ldots, n\right\}$, encompasses the entire set $A$ as $|J|=\sum_{i=2}^{k} s_i - s_1 \leq \sum_{i=1}^{k-1} s_i$ (due to $s_k\leq 2\times s_1$). As a result, $a_1\in \{1, \ldots, s_1\}$. Nevertheless, in that case, $$a_k-a_1\geq (n-\alpha+1)-(m) > n-(\alpha+m)=n-s_k,$$ leading to a contradiction. Therefore, c is a proper coloring, and the number of colors is used in this coloring is $n - \left(\sum_{i=2}^{k} s_i\right) = n - \left(\sum_{i=1}^{k-1} s_i\right) - \alpha$, as needed. This concludes the proof.

\end{proof}

\begin{lemma}\label{Lemma:equivalence}
Let $n, k$ be positive integers and $s_1, \ldots, s_{k-1}$ be fixed positive integers such that $s_i \geq 2$ for all $1 \leq i \leq k-1$. Let $m = \min\{s_1, \ldots, s_{k-1}\}$. Then the following statements are equivalent:
\begin{enumerate}
    \item[(a)] for all $n \geq \sum_{i=1}^{k} s_i$,
    \[
    \chi\left(\mathrm{KG}^{2}(n, k)_{(s_1,\ldots, s_{k-1}, s_k)\textup{-stab}}\right) = n - \sum_{i=1}^{k-1} s_i,
    \]
    where $s_k=m$.
    \item[(b)] For all $s_k$ and all $n$ with $s_k \leq 2m$ and $n \geq \sum_{i=1}^{k} s_i$, we have
    \[
    \chi\left(\mathrm{KG}^{2}(n, k)_{(s_1,\ldots, s_{k-1}, s_k)\textup{-stab}}\right) = n - \sum_{i=1}^{k-1} s_i - \max\{0, s_k - m\}.
    \]
\end{enumerate}
\end{lemma}

\begin{proof}
The implication (b) $\Rightarrow$ (a) is immediate, as case $s_k = m$ is a special case of (b).

Now assume that (a) holds. Let $n \geq \sum_{i=1}^{k} s_i$ and let $s_k = m + \alpha$ for some integer $\alpha$. The hypothesis of (b) requires $\alpha \leq m$.
If $\alpha \leq 0$, then the result follows from (a) and the fact that $\mathrm{KG}^{2}(n, k)_{(s_1,\ldots, s_{k-1}, s_k)\textup{-stab}}$ contains $\mathrm{KG}^{2}(n, k)_{(s_1,\ldots, s_{k-1}, m)\textup{-stab}}$ as its subgraph and Lemma~\ref{lemma: coloring} . Assume now that $0 < \alpha \leq m$. The upper bound,
\[
\chi\left(\mathrm{KG}^{2}(n, k)_{\vec{s}\textup{-stab}}\right) \leq n - \sum_{i=1}^{k-1} s_i - \alpha,
\]
is provided by the construction in Lemma~\ref{Lemma: 1 coloring}. For the matching lower bound, consider the subgraph $\mathrm{KG}^{2}(n-\alpha, k)_{(s_1,\ldots, s_{k-1}, m)\mathrm{-stab}}$. This is a subgraph of $\mathrm{KG}(n, k)_{\vec{s}\textup{-stab}}$. Applying statement (a) to this subgraph yields:
\[
\chi\left(\mathrm{KG}^{2}(n, k)_{\vec{s}\textup{-stab}}\right) \geq \chi\left(\mathrm{KG}^{2}(n-\alpha, k)_{(s_1,\ldots, s_{k-1}, m)\mathrm{-stab}}\right) = (n-\alpha) - \sum_{i=1}^{k-1} s_i.
\]
This establishes the required lower bound and completes the proof of (b).
\end{proof}

\begin{proof}[\textbf{Proof of Theorem~\ref{Theorem:Main2} part (i)}]
BY Lemma~\ref{Lemma:equivalence}, we can assume $s_k=m=2$ and then the result comes from Theorem~\ref{Theorem: Main of Hamid & Jozsef}.
\end{proof}
To prove the second part, we employ a combinatorial version of Tucker's Lemma\cite{matousek2004}. This lemma, which is a consequence of the Borsuk-Ulam theorem, is well-established as a powerful tool in combinatorics~\cite{matousek2004, palvolgyi2009, daneshpajouh2018new}. For further background on the applications of the Borsuk-Ulam theorem and its consequences in graph coloring, we refer the reader to the recent survey by~\cite{daneshpajouh2025box}.

For $A \in \{-1, 0, 1\}^n$, the support of $A$ is defined as $\text{supp}(A) = \{i \in [n] \mid A^i \ne 0\}$. The positive and negative components are denoted as $A^+ = \{i \in [n] \mid A^i = 1\}$ and $A^- = \{i \in [n] \mid A^i = -1\}$, respectively. The sign of $A$, represented by $\text{sgn}(A)$, corresponds to the initial non-zero coordinate of $A$. For $A, B \in \{-1, 0, 1\}^n$, the notation $A \preceq B$ is employed if $A^i \ne 0$ consequently implies $A^i = B^i$ universally for all $i \in [n]$. Lastly, fix an arbitrary complete order $\preceq$ on subsets of $[n]$ such that if $|A|<|B|$ then $A\prec B$. For any non-zero $A \in \{-1, 0, 1\}^n$, its alternating number, denoted as $\text{Alt}(A)$, is the largest subset $\{i_1 < i_2 < \dots < i_\ell\} \subseteq \text{supp}(A)$ (based on the order $\prec$) that alternates in sign, characterized by $A^{i_j} \cdot A^{i_{j+1}} = -1$ for every $j = 1, \dots, \ell-1$.
 
\begin{lemma}[Tucker's Lemma]\label{lem:tucker}
Let $\lambda: \{-1, 0, 1\}^n \setminus \{\mathbf{0}\} \to \{\pm 1, \pm 2, \dots, \pm m\}$ be a function satisfying the following two properties:
    \begin{enumerate}
        \item $\lambda(-A) = -\lambda(A)$ for all $A \in \{-1, 0, 1\}^n \setminus \{\mathbf{0}\}$, and
        \item If $A_1 \preceq A_2$ and $|\lambda(A_1)| = |\lambda(A_2)|$, then $\lambda(A_1) = \lambda(A_2)$.
    \end{enumerate}
    Then $m \geq n$.
\end{lemma}
We now establish the second part of Theorem~\ref{Theorem:Main2}, using a key idea from~\cite{chen2015multichromatic}.

\begin{proof}[\textbf{Proof of Theorem~\ref{Theorem:Main2} part (ii)}]
Again, by Lemma~\ref{Lemma:equivalence}, we may assume without loss of generality that $s_k = m$. The upper bound
\[
\chi\left(\mathrm{KG}^{2}(n, k)_{\vec{s}\textup{-stab}}\right) \leq n - \sum_{i=1}^{k-1} s_i
\]
has been established in Lemma~\ref{lemma: coloring}. It remains to prove the matching lower bound. Let $t = \chi\left(\mathrm{KG}^{2}(n, k)_{\vec{s}\textup{-stab}}\right)$, and let $c$ be a proper coloring of the graph using the color set $[t]$. Define the parameters
\[
\alpha = \sum_{i=1}^{k} s_i \quad \text{and} \quad m = t + \sum_{i=1}^{k-1} s_i.
\]
We will construct a function $\lambda : \{-1, 0, 1\}^n \setminus \{\mathbf{0}\} \to \{\pm 1, \dots, \pm m\}$ satisfying the conditions of Tucker's Lemma (Lemma~\ref{lem:tucker}). The function is defined in two cases based on the alternating number $\mathrm{Alt}(A)$ of a non-zero signed set $A$.

\begin{description}
    \item[Case 1: $|\mathrm{Alt}(A)| \leq \alpha$.] 
        In this case, we define
        \[
        \lambda(A) = \mathrm{sgn}(A) \cdot |\mathrm{Alt}(A)|.
        \]
        Note that $\lambda(A) \in \{\pm 1, \dots, \pm \alpha\}$.

    \item[Case 2: $|\mathrm{Alt}(A)| > \alpha$.] 
        Let $\mathrm{Alt}(A) = \{i_1 < i_2 < \dots < i_\ell\}$. We construct $s_k$ pairwise disjoint $\vec{s}$-stable $k$-sets from the initial segment of $\mathrm{Alt}(A)$. For $j = 1, \dots, s_k$, define
        \[
        F_j = \left\{ i_j, i_{j + s_1}, i_{j + s_1 + s_2}, \dots, i_{j + \sum_{h=1}^{k-1} s_h} \right\}.
        \]
        Since the values $s_1, \dots, s_{k-1}$ are even, each set $F_j$ is entirely contained in either $A^+$ or $A^-$. Let $F_{j_0}$ be a set achieving the minimum color among these, i.e., $c(F_{j_0}) = \min \{ c(F_1), \dots, c(F_{s_k}) \}$. We then define
        \[
        \lambda(A) = \mathrm{sgn}(F_{j_0}) \cdot \left(\alpha + c(F_{j_0})\right),
        \]
        where $\mathrm{sgn}(F_{j_0}) = +1$ if $F_{j_0} \subseteq A^+$ and $-1$ if $F_{j_0} \subseteq A^-$. Note that since $1\leq c(F_{j_0}) \leq t - s_k$, we have $\lambda(A) \in \{\pm(\alpha+1), \dots, \pm m\}$.
\end{description}
Let us now proceed to verify the conditions of Tucker's lemma. It is straightforward to observe the first property from the definition of $\lambda$. Thus, let us see why $\lambda$ fulfills the second property. If $A_1 \preceq A_2$ and $|\lambda(A_1)| = |\lambda(A_2)| \le \alpha$, then $|\mathrm{Alt}(A_1)| = |\mathrm{Alt}(A_2)|$. Consequently, $\mathrm{sgn}(A_1) = \mathrm{sgn}(A_2)$ as otherwise we would have $|\mathrm{Alt}(A_2)| > |\mathrm{Alt}(A_1)|$. If $|\lambda(A_1)| = |\lambda(A_2)| > \alpha$, then $\vec{s}$-stable $k$-subsets $F_1 \subseteq A_1^{\mathrm{sgn}(F_1)}$ and $F_2 \subseteq A_2^{\mathrm{sgn}(F_2)}$ exist such that $c(F_1) = c(F_2)$. If $\mathrm{sgn}(F_1) \ne \mathrm{sgn}(F_2)$, then $A_1^{\mathrm{sgn}(F_1)} \cap A_2^{\mathrm{sgn}(F_2)} = \emptyset$, which implies that $F_1$ and $F_2$ are disjoint. This would contradict the properness of the coloring $c$. Therefore, $\lambda(A_1) = \lambda(A_2)$, and all conditions are duly satisfied. Hence, $m = t + \sum_{i=1}^{k-1} s_i \ge n$, or equivalently, $t \ge n - \sum_{i=1}^{k-1} s_i$, as asserted. The theorem is thus established.

\end{proof}

\subsection{Proof of Theorem~\ref{Theorem: Main3}}
To establish our final result, we require several auxiliary concepts and two lemmas. Let $n, s_1, s_2$ be positive integers. We define the graph $W(n, s_1, s_2)$ to have the vertex set $[n]$, where two distinct vertices $i < j$ are adjacent if and only if the circular distance between them satisfies $s_1 \leq |i - j| \leq n - s_2$. A crucial observation is that the Kneser graph $\mathrm{KG}(n,2)_{(s_1,s_2)\mathrm{-stab}}$ is isomorphic to the complement of the line graph of $W(n, s_1, s_2)$. This connection allows us to translate the problem of coloring the stable Kneser graph into a problem of partitioning the edge set of $W(n, s_1, s_2)$. In general, proper colorings of $G$ are exactly partitions of the edge set of $H$ into stars and triangles. A {\em star} is a tree with a vertex, the {\em center} of the star, connected to all other vertices. A single-edge tree is in particular a star. In this particular case, we will always assume that exactly one of the two vertices has been identified as the center, so as to be in a position to always speak of the center of a star without ambiguity. A {\em triangle} is a circuit of length $3$. We call such a partition of $H$ into stars and triangles an {\em ST-partition}. Any element in an ST-partition is called a {\em part}. This viewpoint has been successfully developed and applied for solving similar questions in~\cite{daneshpajouh2021colorings} and later in~\cite{daneshpajouh2024number}, which has turned out to be indeed helpful for this case as well.
   
\begin{lemma}
Let $n, s_1, s_2$ be positive integers such that $s_2 \geq 2s_1$, and $n \geq 2s_2 - 2$. Then
\[
\alpha(W(n, s_1, s_2)) = 2s_1.
\]
\end{lemma}

\begin{proof}
We first prove the lower bound $\alpha(W(n, s_1, s_2)) \geq 2s_1$. Consider the set
\[
S_0 = \{1, 2, \dots, s_1\} \cup \{n - s_1 + 1, n - s_1 + 2, \dots, n\}.
\]
This set has cardinality $2s_1$. We claim $S_0$ is an independent set in $W(n, s_1, s_2)$. For any two distinct vertices $i, j \in S_0$, the distance $|i - j|$ is either less than $s_1$ (if both are in the initial segment or both in the terminal segment) or strictly greater than $n - s_2$ (if one is in the initial and the other in the terminal segment, since $|j-i| \geq (n-s_1+1)-s_1\geq  s_2 +1 > n-s_2$ as our hypothesis says $s_2 \geq 2s_1$). In both cases, the pair $\{i, j\}$ is not an edge by the definition of $W(n, s_1, s_2)$. Hence, $S_0$ is independent, which establishes the lower bound. For the upper bound, let $S$ be an independent set of the maximum size. By symmetry, assume $1 \in S$. The vertex $1$ is adjacent to all $v$ satisfying $s_1 \leq |1-v| \leq n-s_2$. Therefore, $S\subseteq \{1,\ldots,s_1\} \cup \{n-s_2+2,\ldots,n\}$.

Suppose for contradiction that $|S \cap \{n-s_2+2,\ldots,n\}| \geq s_1+1$. Let $a_1 < \ldots < a_{s_1+1}$ be the $s_1+1$ smallest elements in this intersection. Then:
\[
s_1 \leq a_{s_1+1} - a_1 \leq n - (n-s_2+2) = s_2 - 2 \leq n - s_2,
\]
where the last inequality uses the hypothesis $n \geq 2s_2-2$. This implies that $a_1$ and $a_{s_1+1}$ are adjacent, a contradiction. Hence, $|S \cap \{n-s_2+2,\ldots,n\}| \leq s_1$, and since $|\{1,\ldots,s_1\}| = s_1$, we conclude $|S| \leq 2s_1$.
\end{proof}

\begin{lemma}\label{Lemma:  Butterfly}
There is no induced butterfly ( a union of two triangles sharing exactly one vertex) in $W(n, s_1, s_2)$.
\end{lemma}
\begin{proof}
Assume for contradiction that there is an induced butterfly in $W(n, s_1, s_2)$. Let the central vertex be $v$. Note that the graph $W(n, s_1, s_2)$ is rotationally symmetric because its definition depends only on circular distance. Therefore, by applying a cyclic rotation to the vertex labels, we can assume $v = 1$. Then the other vertices are $a, b, c, d$ with $1 < a < b < c < d \leq n$. Now, we consider two cases.

\textbf{Case I:} Suppose $1$, $a$, and $d$ form one of the triangles. In this case, $a$ must be connected to $c$, which contradicts the fact that the two triangles (with vertices $\{1, a, d\}$ and $\{1, b, c\}$) form an induced butterfly.

\textbf{Case II:} Suppose $a$ and $d$ each belong to different triangles. In this case, it is not hard to see that $a$ must be connected to $d$, which again contradicts the fact that these two triangles form an induced butterfly.
\end{proof}

\begin{proof}[\textbf{Proof of Theorem~\ref{Theorem: Main3}}]
Since $\mathrm{KG}^{2}(n, 2)_{(s_1,s_2)\textup{-stab}}$ is a subgraph of $\mathrm{KG}^{2}(n, 2)_{(s_1,2s_1)\textup{-stab}}$, and the latter is $(n-2s_1)$-colorable, it follows that $\mathrm{KG}^{2}(n, 2)_{(s_1,s_2)\textup{-stab}}$ is also $(n-2s_1)$-colorable. To show that $n-2s_1$ colors are necessary, we use the ST-partition framework. Let $\mathcal{P}$ be an optimal ST-partition of $G = W(n, s_1, s_2)$ with the minimum number of triangles. By \cite[Lemma 1]{daneshpajouh2021colorings}, $\mathcal{P}$ satisfies:
\begin{enumerate}
    \item[(1)] No vertex of a triangle in $\mathcal{P}$ is the center of a star in $\mathcal{P}$.
    \item[(2)] If any circuit $C$ of $G$ having no edge of
any star in $\mathcal{P}$, then $C$ is a triangle in $G$.
\end{enumerate}
Now, we consider two cases.
\\~\\
\noindent\textbf{Case I:} If there are no triangles in $\mathcal{P}$, then at most $\alpha(G)$ vertices of $G$ can be non-centers, since any set of $\alpha(G)+1$ vertices contains at least one adjacent pair. Thus, $|\mathcal{P}| \geq |V(G)| - \alpha(G)$. Moreover, by Lemma~\ref{Lemma: Butterfly}, $\alpha(G) = 2s_1$, and therefore $|\mathcal{P}| \geq n - 2s_1$.
\\~\\
\noindent\textbf{Case II:} Now, consider the case where $\mathcal{P}$ contains at least one triangle. Fix a triangle $T$ in $\mathcal{P}$ with vertices $1 \leq a < b < c \leq n$. By symmetry, we may assume $a = 1$ (via a cyclic rotation of the labels, if necessary). Let $D$ be the set of vertices in $G$ that are neither centers of stars nor belong to $T$. We claim that no vertex in $D$ is adjacent to any vertex in $T$. Suppose, for contradiction, that such an edge exists. This edge must belong to some part of $\mathcal{P}$. It cannot be an edge of a star, as this would violate property (1) (a vertex of $T$ cannot be a star center). Therefore, it must be an edge of another triangle $T'$. Since $G$ is butterfly-free (by Lemma~\ref{Lemma:  Butterfly}), there must be a second edge between $T$ and $T'$. This second edge must also belong to a triangle, which would necessarily be a third triangle by property (1). The existence of these three triangles would form a circuit that avoids star edges yet is not a triangle, contradicting property (2). This proves the claim.

We now show that $|D| \leq 2s_1-2$. Assume, for contradiction, that $|D| \geq 2s_1-1$. Note that all vertices in $D$ must lie on the arc $c1$ (the clockwise arc from $c$ to $1$). Let $a_1, \ldots, a_{2s_1-1}$ be the first $2s_1-1$ vertices of $D$ encountered along this arc respectively. A straightforward verification shows that the vertex $a_{s_1}$ must be adjacent to $b$, which contradicts the established non-adjacency between $D$ and $T$.
 Therefore, the total number of parts is bounded by:
\[
|\mathcal{P}|\geq n-3+1-|D|\geq n-2-(2s_1-2)\geq n-2s_1,
\]
which completes the proof.

\end{proof}

\begin{remark}
One might conjecture that Theorem~\ref{Theorem: Main3} suggests that the chromatic number stabilizes once $s_2 \geq 2s_1$. However, this is not the case. For example, the graph $G = \mathrm{KG}^{2}(6, 2)_{(1,4)\textup{-stab}}$ has chromatic number $3$, not $4$. A lower bound of $3$ is given by the clique forms with the vertices $\{1,2\}, \{3,4\}, \{5,6\}$, and an upper bound is provided by the following proper $3$-coloring:
\begin{align*}
C_1 &= \{\{1,2\}, \{1,3\}, \{2,3\}\}, \
C_2 &= \{\{2,4\}, \{3,4\}, \{4,5\}, \{4,6\}\}, \
C_3 &= \{\{3,5\}, \{5,6\}\}.
\end{align*}
\end{remark}

\section{Open Problems}

Although a general formula for the chromatic number of $\mathrm{KG}(n,k)_{\vec{s}\textup{-stab}}$ remains elusive, Theorem~\ref{Theorem:Main2} motivates the following conjecture.

\begin{conjecture}\label{conj5}
Let $n, k$ be positive integers with $k \geq 2$, and let $\vec{s} = (s_1, \ldots, s_k)$ be a positive integer vector such that $s_i \geq 2$ for all $i = 1, \ldots, k-1$. If $n \geq \sum_{i=1}^{k} s_i$ and $s_k \leq 2 \cdot \min\{s_i : 1 \leq i \leq k-1\}$, then
\[
\chi\left(\mathrm{KG}^{2}(n, k)_{\vec{s}\textup{-stab}}\right) = n - \left(\sum_{i=1}^{k-1} s_i\right) - \max\left\{0, s_k - \min\{s_i : 1 \leq i \leq k-1\}\right\}.
\]
\end{conjecture}

\begin{remark}
By Lemma~\ref{Lemma:equivalence}, verifying Conjecture~\ref{conj5} reduces to the case $s_k = \min\{s_i : 1 \leq i \leq k-1\}$. Thus, Conjecture~\ref{conj5} holds for $k=2$, as Conjecture~\ref{Conj: Meu} is known to be true for $r=k=2$ and arbitrary $s\geq 2$~\cite{daneshpajouh2021colorings}.
\end{remark}

In light of Theorem~\ref{Thm:main1} and Conjecture~\ref{Conj: Meu}, we pose the following question.

\begin{question}\label{OpenProblem}
For which parameter vectors $\vec{s}=(s_1,\ldots, s_k) \in \mathbb{N}^k$, integers $r\geq 2$, and $k\geq 1$ does the following equality hold?
\[
\chi\left(\mathrm{KG}^{r}(n,k)_{\vec{s}\textup{-stab}}\right) = \left\lceil\frac{n - \max\{r(k-1), \sum_{i=1}^{k-1}s_i\}}{r-1}\right\rceil?
\]
\end{question}

Our proof of Theorem~\ref{Thm:main1} yields a partial answer to Question~\ref{OpenProblem}:

\begin{theorem}\label{Thm:VectorStable}
For all integer vectors $\vec{s}=(s_1,\ldots, s_{k-1},2)$ and all $m, n\in\mathbb{N}$ with $n\geq 2^m(\sum_{i=1}^{k-1}s_i+2)$ and $s_i\geq 2$ for $1\leq i\leq k-1$, we have
\[
\chi\left(\mathrm{KG}^{2^{m+1}}(n,k)_{(s_1\times 2^{m},\ldots,s_{k-1}\times 2^{m}, 2^{m+1})\textup{-stab}}\right) = \left\lceil\frac{n - 2^m\times(\sum_{i=1}^{k-1}s_i)}{2^{m+1}-1}\right\rceil.
\]
\end{theorem}

\medskip
 
\bibliographystyle{alpha}
\bibliography{main}

\end{document}